\documentclass[11pt,dvipdfmx]{amsart}
\usepackage{amsmath,amscd,latexsym}
\usepackage{amssymb}
\usepackage{amsbsy}
\usepackage{amsfonts}
\usepackage[all]{xy}
\usepackage{color}
\usepackage[dvipsnames]{xcolor}
\usepackage{graphicx}
\usepackage{subcaption}
\usepackage[abs]{overpic}
\usepackage{multirow}
\usepackage{array}
\usepackage{paralist}
\usepackage{titletoc}
\usepackage{url}
\usepackage{comment}

\theoremstyle{plain}
\newtheorem{theorem}{Theorem}[section]
\newtheorem{proposition}[theorem]{Proposition}
\newtheorem{lemma}[theorem]{Lemma}

\newtheorem{remark}[theorem]{Remark}
\newtheorem{definition}[theorem]{Definition}
\newtheorem{notation}[theorem]{Notation}

\newtheorem{main theorem}[theorem]{Main Theorem}

\newtheorem{claim}[theorem]{Claim}

\newcommand{\interior}{\operatorname{int}}

\newcommand{\length}{\operatorname{\ell}}

\newcommand{\Lk}{\operatorname{Lk}}

\newcommand{\RR}{\mathbb{R}}

\newcommand{\EE}{\mathbb{E}}

\newcommand{\clk}{\mathrm{lk}} 
\newcommand{\ost}{\mathrm{st}}
\newcommand{\cst}{\mathrm{St}}

\newcommand{\BB}{\mathcal{B}}

\makeatletter
\renewcommand\subsection{\@startsection{subsection}{2}{0mm}
    {-10.5dd plus-8pt minus-4pt}{10.5dd}
     {\normalsize\upshape}}
\makeatother

\begin{document}
\title[Convexity in CAT(0) cubed complexes]
{Combinatorial local convexity implies convexity
in finite dimensional CAT(0) cubed complexes}
\author{Shunsuke Sakai}
\address{Gifu Higashi High School\\
4-17-1, noisshiki, Gifu City
500-8765, Japan}
\email{shunsuke463@gmail.com}

\author{Makoto Sakuma}
\address{Advanced Mathematical Institute\\
Osaka Metropolitan University\\
3-3-138, Sugimoto, Sumiyoshi, Osaka City
558-8585, Japan}
\address{Department of Mathematics\\
Faculty of Science\\
Hiroshima University\\
Higashi-Hiroshima, 739-8526, Japan}
\email{sakuma@hiroshima-u.ac.jp}


\begin{abstract}
We give a proof of the following theorem,
which is well-known among experts:
A connected subcomplex $W$ of 
a finite dimensional CAT(0) cubed complex $X$ is convex
if and only if $\Lk(v,W)$ is a full subcomplex of $\Lk(v,X)$
for every vertex $v$ of $W$.
\end{abstract}

\maketitle

\section{introduction}
\label{sec:introduction}

The purpose of this note is to give a proof of the following theorem,
which is well-known among experts.

\begin{theorem}
\label{thm:convex}
Let $X$ be a finite dimensional
CAT(0) cubed complex and $W$ a connected subcomplex of $X$. 
Then $W$ is convex in $X$ if and only if it satisfies the condition (CLC) below:

{\rm (CLC)} $\Lk(v,W)$ is a full subcomplex of $\Lk(v,X)$
for every vertex $v$ of $W$.
\end{theorem}

Recall that a subcomplex $K$ of a simplicial complex $L$
is {\it full} if any simplex of $L$
whose vertices are in $K$ is in fact entirely contained in $K$.
The condition (CLC) is nothing other than 
the definition for $W$ to be
\lq\lq combinatorially locally convex'' in $X$, 
in the sense of Haglund-Wise~\cite[Definition 2.9]{Haglund-Wise}
(cf.\ Haglund~\cite[Definitions 2.8 and 2.9]{Haglund}).
(Their terminology does not contain the adjective combinatorial.)
In fact, they introduced the concept of a \lq\lq combinatorial local isometry'',
and define $W$ to be combinatorially locally convex in $X$
if the inclusion map $j:W\to X$ is a combinatorial local isometry.
As (implicitly) suggested in~\cite{Stack},
Theorem~\ref{thm:convex} is an immediate consequence
of~\cite[Lemma 2.11]{Haglund-Wise} concerning 
combinatorial local isometries from cube complexes
to finite dimensional non-positively curved cube complexes.

In~\cite[Proof of Lemma 2.11]{Haglund-Wise}
appealing to~\cite[Proposition II.4.14]{BH}
(which is deduced from the classical Cartan-Hadamard theorem),
it is implicitly assumed that
a combinatorial local isometry
is a local isometry in the usual sense
(Definition~\ref{def:convex}(2)).
On the other hand,
Haglund writes 
in~\cite[the paragraph preceding Theorem 2.13]{Haglund}
that in the finite dimensional case
it can be checked that combinatorial local isometries are 
precisely local isometries of the $\ell_2$ (Euclidean) metrics.
Moreover, Petrunin notes in~\cite{Stack} that
combinatorial local convexity implies local convexity 
and that this can be proved the same way 
as the flag condition (Gromov's link condition) for CAT(0) spaces.
Thus Theorem~\ref{thm:convex} is established 
by~\cite[Lemma 2.11]{Haglund-Wise},
though we could not find a reference
that includes a proof of 
the implicit assertion.\footnote{It turned out that
the assertion immediately follows from the result \cite[Theorem 1(2)]{Crisp-Wiest} due to Crisp and Wiest.
Moreover, Leary's article \cite{Leary} includes a direct proof of 
Theorem~\ref{thm:convex} and so that of the assertion.
See Late Additions (Section \ref{sec:late-addition})
for the details.}

The purpose of this note is to give a full proof of 
Theorem~\ref{thm:convex} by writing down a proof of 
the assertion (Theorem~\ref{prop:local-convex}).
Our proof totally depends on
Bridson-Haefliger~\cite{BH},
and it may be regarded as a relative version of 
the proof of
Gromov's link condition
included in the book
(see~\cite[Proofs of Theorems II.5.2 and II.5.20]{BH}).

The main bulk of this note was originally written as a part of~\cite{SS}.
After learning from~\cite{Stack} that
Theorem~\ref{thm:convex} is well-known among experts
(as we had expected) and
that it is essentially contained in 
Haglund-Wise~\cite[Lemma 2.11]{Haglund-Wise},
we decided to move that part of~\cite{SS}
into this separate note.
We hope this note is of some use to those who are not
so familiar with the relation between
the two concepts concerning local convexity.

We note that
Theorem~\ref{thm:convex} may be regarded
as a Euclidean metric version of the combinatorial result 
by Haglund~\cite[Theorem 2.13]{Haglund},
which shows that combinatorial convexity~\cite[Definition 2.9]{Haglund}
is a local combinatorial property.
However, Theorem~\ref{thm:convex} is weaker 
than~\cite[Theorem 2.13]{Haglund},
in the sense that the former assumes finite dimensionality 
whereas the latter does not.\footnote{Leary gives a proof 
in the infinite dimensional case, too.}

As is summarized in~\cite{RC},
local convexity implies (global) convexity
in various settings, including the following:
\begin{itemize}
\item[-]
closed connected subsets in a Euclidean space (Nakajima~\cite{Nakajima} and Tietze~\cite{Tietze}),
\item[-]
closed connected subsets (whose diameter is 
less than $\pi/\sqrt{\kappa}$ when $\kappa>0$) 
in a complete CAT($\kappa$) space
(Bux-Witzel~\cite[Theorems 1.6 and 1.10]{Bux-Witzel} 
and Ramos-Cuevas~\cite[Theorem 1.1]{RC}), and 
\item[-]
closed connected (by rectifiable arcs) subsets of
proper Busemann spaces (Papadopoulos~\cite[Proposition 8.3.3]{Pa}).
\end{itemize} 
The following well-known fact is the simplest non-trivial example of such results.
\begin{itemize}
\item[-]
A local geodesic in a CAT($\kappa$) space
(of length less than $\pi/\sqrt{\kappa}$ when $\kappa>0$) 
is a geodesic~\cite[Proposition I.1.4(2)]{BH}.
\end{itemize} 
This fact is repeatedly (though implicitly) used in this note.

\medskip
{\bf Acknowledgement.}
We thank Hirotaka Akiyoshi for his criticism and helpful discussion.
The second author is supported by JSPS KAKENHI Grant Number JP20K03614
and by Osaka Central Advanced Mathematical Institute 
(MEXT Joint Usage/Research Center on Mathematics and Theoretical Physics JPMXP0619217849).

\section{Basic definitions and outline of the proof of Theorem~\ref{thm:convex}}

We first recall basic facts 
concerning non-positively curved spaces following Bridson-Haefliger~\cite{BH}.

Let $X=(X,d)$ be a metric space.
In this paper, we mean by
a {\it geodesic} in $X$ an isometric embedding $g:J\to X$
where $J$ is a connected subset of $\RR$.
If $J$ is a closed interval,
we call $g$ a {\it geodesic segment}.
We do not distinguish between a geodesic and its image.
$X$ is a {\it geodesic space}
if every pair of points can be joined by a geodesic in $X$.
It is said to be {\it uniquely geodesic}
if for every pair of points there is a unique geodesic joining them.
For points $a$ and $b$ in a geodesic space $X$,
we denote by $[a,b]$ a geodesic segment joining $a$ and $b$.
The symbols $(a,b)$, $[a,b)$ and $(a,b]$
represent open or half-open geodesic segments, respectively.  
The distance $d(a,b)$ is equal to $\length([a,b])$,
the length of the geodesic segment $[a,b]$.
Thus the geodesic space $X$ is a {\it length space}
in the sense that the distance between two points
is the infimum over the lengths of rectifiable curves 
that join them~\cite[I.1.18 and I.3.1]{BH}.

A geodesic space $X$ is a {\it CAT(0) space}
if any geodesic triangle is thinner than a comparison triangle in 
the Euclidean plane $\EE^2$,
that is, the distance between any points on a geodesic triangle is
less than or equal to the corresponding points on a comparison triangle
~\cite[Definition~II.1.1]{BH}.
A CAT(0) space is uniquely geodesic~\cite[Proposition~II.1.4(1)]{BH}.
A geodesic space $X$ is said to be {\it non-positively curved}
if it is locally a CAT(0) space,
i.e., for every $x\in X$ there exists $r>0$
such that the open $r$-ball 
$B_X(x,r):=\{y\in X \ | \ d(x,y)<r\}$
in $X$ with center $x$,
endowed with the induced metric, 
is a CAT(0) space~\cite[Definition~II.1.2]{BH}.

A {\it cubed complex} is  
a metric space $X=(X,d)$ obtained from a disjoint union of unit cubes
$\hat X=\bigsqcup_{\lambda\in\Lambda} (I^{n_\lambda}\times\{\lambda\})$
by gluing their faces through isometries.
To be precise, it is an $M_{\kappa}$-polyhedral complex with $\kappa=0$
in the sense of~\cite[Definition I.7.37]{BH}
that is made up of Euclidean unit cubes, i.e.,
the set $\mathrm{Shapes}(X)$ in the definition 
consists of Euclidean unit cubes.
(See~\cite[Example (I.7.40)(4)]{BH}.)
The metric $d$ on $X$ is the length metric induced from the Euclidean metric
of the unit cubes.
See~\cite[I.7.38]{BH} for a precise definition.
Every finite dimensional cubed complex 
is a complete geodesic space~\cite[Theorem in p.97 or I.7.33]{BH},
where the {\it dimension} of the cubed complex is defined to be 
$\max\{n_{\lambda}\}$.
Note that the restriction of
the projection $p:\hat X \to X$ to
$I^{n_\lambda}\times\{\lambda\}$ is not necessarily injective.
Thus a cubed complex is not necessarily a {\it cubical complex} in the sense of \cite[Definition I.7.32]{BH}, 
i.e., a {\it cube complex} which is {\it simple} in the sense of 
\cite{Haglund, Haglund-Wise}.
However, the difference is not essential
for non-positively curved cubed complexes,
because the second cubical subdivision of
a non-positively curved cubed complex is a cubical complex
by \cite[Corollary C.11]{Leary}, and 
because 
the metric of the cubed complex and that of
its cubical subdivision (after rescaling) are identical
(cf.~\cite[Lemma I.7.48]{BH}).
(A {\it cube complex} in \cite[Appendix C]{Leary} is a {\it cubed complex} 
in this note, i.e., in the sense of \cite[Example I.7.40(4)]{BH},
as noted in 
\cite[the first sentence in Appendix C]{Leary}.)

Two non-trivial geodesics issuing from a point $x\in X$
are said to define the {\it same direction}
if the Alexandrov angle between them is zero.
This defines an equivalence relation on the set of 
non-trivial geodesics issuing from $x$,
and the Alexandrov angle induces a metric 
on the set of the equivalence classes.
The resulting metric space is called the
{\it space of directions} at $x$ and denoted 
$S_x(X)$~\cite[Definition II.3.18]{BH}.

Suppose $x$ is a vertex $v$ of the cubed complex $X$.
Then 
the space $S_v(X)$ is obtained by gluing the spaces
$\{S_{v_{\lambda}}(I^{n_{\lambda}}\times\{\lambda\})\}_{v_{\lambda}\in p^{-1}(v)}$.
Here $S_{v_{\lambda}}(I^{n_{\lambda}}\times\{\lambda\})$ is the space   
of directions in the cube $I^{n_{\lambda}}\times\{\lambda\}$ at the vertex $v_{\lambda}$;
so it is an {\it all-right spherical simplex},
a geodesic simplex in the unit sphere $S^{n_{\lambda}-1}$  
all of whose edges have length $\pi/2$.
Hence $S_v(X)$ has a structure of a finite dimensional {\it all-right spherical complex}, 
namely an $M_{\kappa}$-polyhedral complex with $\kappa=1$
in the sense of~\cite[Definition I.7.37]{BH}
that is made up of all-right spherical simplices, i.e.,
the set $\mathrm{Shapes}(X)$ in the definition 
consists of all-right spherical simplices.
This complex is called the {\it geometric link} of $v$ in $X$,
and is denoted by $\Lk(v,X)$~\cite[(I.7.38)]{BH}.
It should be noted that $\Lk(v,X)$ is not necessarily a simplicial complex:
it is a simplicial complex if $X$ is a cubical complex.
The geometric link $\Lk(v,X)$ is endowed with the length metric $d_{\Lk(v,X)}$
induced from the spherical metrics of the all-right spherical simplices.
Let $d_{\Lk(v,X)}^{\pi}$ be the metric defined by
\[
d_{\Lk(v,X)}^{\pi}(u_1,u_2):=\min\{d_{\Lk(v,X)}(u_1,u_2),\pi\}.
\]
Then the metric $d_{S_v(X)}$ on $S_v(X)=\Lk(v,X)$ 
is equal to the metric $d_{\Lk(v,X)}^{\pi}$
(see~\cite[the second sentence in p.191]{BH} 
or~\cite[Lemma 5.5]{SS}).

\begin{definition}
\label{def:convex}
{\rm
Let $X$ be a uniquely geodesic space and $W$ a subset of $X$.

(1) $W$ is {\it convex} in $X$
if, for any distinct points $a$ and $b$ in $W$,
the unique geodesic segment $[a,b]$ in $X$ is contained in $W$.

(2) $W$ is {\it locally convex} in $X$ if, for every $x\in W$,
there is an $\epsilon >0$ such that
$W\cap B_X(x,\epsilon)$ is convex in $X$,
where $B_X(x,\epsilon)$
is the open $\epsilon$-ball in $X$
with center $x$.

(3) Assume that
$X$ is a cubed complex 
and $W$ is a subcomplex of $X$.
Then $W$ is {\it combinatorially locally convex} in $X$
if it satisfies the the condition (CLC),
i.e.,
$\Lk(v,W)$ is a full subcomplex of $\Lk(v,X)$
for every vertex $v$ of $W$.
}
\end{definition}

In the next section, we prove the following theorem.

\begin{theorem}
\label{prop:local-convex}
Let $X$ be a finite dimensional
CAT(0) cubed complex and $W$ a subcomplex of $X$. 
Then $W$ is locally convex in $X$ if and only if it is
combinatorially locally convex in $X$.
\end{theorem}

In the reminder of this section,
we give a proof of Theorem~\ref{thm:convex} by using the above theorem
and following~\cite[the proof of Lemma 2.11]{Haglund-Wise}.
The starting point of the proof is the following version of
the Cartan-Hadamard theorem.

\begin{proposition}
\label{prop:CH0}
{\rm \cite[Special case of Theorem II.4.1(2)]{BH}}
Let $X$ be a complete, connected, geodesic space.
If $X$ is non-positively curved, then
the universal covering $\tilde X$ (with the induced length metric)
is a CAT(0) space.
\end{proposition}

See~\cite[Definition I.3.24]{BH}  
for the definition of the {\it induced length metric} on $\tilde X$.
The Cartan-Hadamard theorem implies the following result
~\cite[Proposition II.4.14]{BH},
which plays an essential role
in~\cite[Proof of Lemma 2.11]{Haglund-Wise}
and so in the proof of Theorem~\ref{thm:convex}.

\begin{proposition}
\label{prop:CH}
{\rm \cite[Proposition II.4.14]{BH}}
Let $X$ and $Y$ be a complete, connected metric space.
Suppose that $X$ is non-positively curved and that $Y$ is locally a length space.
If there is a map $f:Y\to X$ that is locally an isometric embedding,
then $Y$ is non-positively curved and:
\begin{enumerate}[\rm (1)]
\item
For every $y_0\in Y$,
the homomorphism $f_*:\pi_1(Y,y_0)\to \pi_1(X, f(y_0))$
induced by $f$ is injective.
\item
Consider the universal coverings $\tilde X$ and $\tilde Y$ with 
their induced length metrics.
Every continuous lifting $\tilde f: \tilde Y \to \tilde X$ of $f$
is an isometric embedding.
\end{enumerate}
\end{proposition}

In the above proposition, 
$f:Y\to X$ being {\it locally an isometric embedding}
means that, for every $y\in Y$,
there is an $\epsilon>0$ such that the restriction of $f$
to the open $\epsilon$-ball $B_Y(y,\epsilon)$ in $Y$
is an isometry onto 
its image in $X$~\cite[the sentence preceding Proposition II.4.14]{BH}.

\medskip

We now give a proof of Theorem~\ref{thm:convex}
following~\cite[Proof of Lemma 2.11]{Haglund-Wise}
and assuming Theorem~\ref{prop:local-convex}.

\begin{proof}[Proof of Theorem~\ref{thm:convex}]
Let $X$ be a finite dimensional
CAT(0) cubed complex and $W$ a connected subcomplex of $X$. 
Suppose $W$ is combinatorially locally convex.
Then
$W$ is locally convex by Theorem~\ref{prop:local-convex}.

\begin{claim}
\label{claim:local-convexity}
The inclusion map
$i:W\to X$, regarded as a map between cubed complexes,
is locally an isometric embedding,
namely, 
for every $x\in W$,
there is an $\epsilon>0$ such that the restriction of $j$
to the open $\epsilon$-ball $B_W(x,\epsilon)$ in $W$
(with respect to the metric $d_W$ 
of the cubed complex $W$)
is an isometry onto its image in the cubed complex $X$.
\end{claim} 

\begin{proof}
Let $\epsilon>0$ be such that
$W\cap B_X(x,\epsilon)$ is convex in $X$.
Then for any $a, b\in W\cap B_X(x,\epsilon)$,
the geodesic $[a,b]$ in $X$ is contained in $W\cap B_X(x,\epsilon)$.
By the definitions of $d_X$ and $d_W$
as length metrics induced from the Euclidean metrics of the unit cubes,
we see that 
$[a,b]$ is also a geodesic in $W$ and $d_X(a,b)=d_W(a,b)$.
Hence the restriction of $i:W \to X$ to
the subspace $W\cap B_X(x,\epsilon) \subset W$
is an isometry onto its image $W\cap B_X(x,\epsilon) \subset X$.
The above observation also implies that 
$W\cap B_X(x,\epsilon)\subset B_W(x, \epsilon)$.
Since 
$B_W(x, \epsilon)\subset W\cap B_X(x,\epsilon)$ obviously holds,
we have $W\cap B_X(x,\epsilon)=B_W(x, \epsilon)$.
Hence, the restriction of $i:W \to X$ to the subspace 
$B_W(x, \epsilon) \subset W$ is an isometry onto its image in $X$.
\end{proof}

Since both $X$ and $W$ are complete~\cite[Theorem in p.97 or I.7.33]{BH}
and since $(W,d_W)$ is a length metric space,
Claim~\ref{claim:local-convexity} enables us to
apply Proposition~\ref{prop:CH} (\cite[Proposition II.4.14]{BH})
to $i:W\to X$,
and so the following hold.
\begin{enumerate}
\item[(0)]
$W$ is non-positively curved.
\item
$i_*:\pi_1(W)\to \pi_1(X)$ is injective.
\item
Consider the universal coverings $\tilde X$ and $\tilde W$
with their induced length metrics.
Every continuous lifting $\tilde i:\tilde W\to \tilde X$ of $i$
is an isometric embedding. 
\end{enumerate}
Since $X$ is a CAT(0) space, $\pi_1(X)=1$ and so $\pi_1(W)=1$ by 
the conclusion (1).
Thus $W=\tilde W$ and it is a CAT(0) space
by the conclusion (0) and the Cartan-Hadamard theorem (Proposition~\ref{prop:CH0}).
Hence, by the conclusion (2), 
$i:W \to X$ is an isometric embedding
of the cubed complex $W=\tilde W$ into the cubed complex $X=\tilde X$.
Thus, for any $a,b\in W$, the unique geodesic $[a,b]$ in 
the CAT(0) space $W$
is also a geodesic in $X$.
This means that $W=i(W)$ is convex in $X$,
completing the proof of the if part.

The only if part immediately follows from
the only if part of Theorem~\ref{prop:local-convex}.
\end{proof}

\begin{remark}
\label{rem:simple-proof}
{\rm
(1) In~\cite[Proof of Proposition II.4.14]{BH},
the proof of the assertion that $Y$ 
is non-positively curved
is rather involved, because it only assumes that 
the complete metric space $Y$ 
is locally a length space.
However, in our setting $Y=W$ is a connected subcomplex of 
the CAT(0) cubed complex which is combinatorially locally convex.
So, the assertion in our setting is an immediate consequence of 
Gromov's link condition~\cite[Theorem II.5.20]{BH}
(cf.\ Lemma~\ref{lem:W-CAT(0)}(2)).

(2) If we appeal to the relatively new results by
Bux-Witzel~\cite[Theorems 1.6 and 1.10]{Bux-Witzel} 
and Ramos-Cuevas~\cite[Theorem 1.1]{RC},
which in particular imply that
a closed connected subset of a complete CAT(0) space is convex
if and only if it is locally convex,
then Theorem~\ref{thm:convex} immediately follows from
Claim~\ref{claim:local-convexity}.
}
\end{remark}

\section{Proof of Theorem~\ref{prop:local-convex}}
\label{sec:Proof-local-convex}

We begin by recalling basic properties of CAT(1) spaces.
A metric space $L=(L,d)$ is a {\it CAT(1) space} if it is a geodesic space
all of whose geodesic triangles of perimeter less than $2\pi$ are not thicker than 
its comparison triangle
in the $2$-sphere $S^2$~\cite[Definition~II.1.1]{BH}.

\begin{proposition}
\label{thm:CAT1}
{\rm (1) (\cite[Theorem II.5.4]{BH})}
Any CAT(1) space is uniquely $\pi$-geodesic,
namely, for any points $a$ and $b$ of the space with $d(a,b)<\pi$,
there is a unique geodesic $[a,b]$ joining $a$ to $b$.

{\rm (2) (\cite[Theorem II.5.18]{BH})}
A finite dimensional all-right angled spherical complex 
is CAT(1) if and only if it is a flag complex.
\end{proposition}

Recall that a {\it flag complex} is a {\it simplicial} complex
in which every finite set of vertices that is pairwise joined by an edge
spans a simplex.

\begin{definition}
\label{def:cone}
{\rm
(\cite[Definition I.5.6]{BH})
For a metric space $Y=(Y,d_Y)$,
the {\it $0$-cone (or the Euclidean cone) $C_0(Y)$ over $Y$} is the metric space 
defined as follows.
As a set $C_0(Y)$ is obtained from $[0,\infty)\times Y$ by 
collapsing $0\times Y$ into a point.
The equivalence class of $(t,y)$ is denoted by $ty$,
where the class of $(0,y)$ is denoted by $0$ and is called the {\it cone point}.
The distance $d(ty,t'y')$ 
between two points $ty$ and $t'y'$ in $C_0(Y)$
is defined by the identity
\[
d(ty,t'y')^2=t^2+{t'}^2-2tt'\cos(d_Y^{\pi}(y,y')),
\]
where $d_Y^{\pi}(y,y')=\min\{d_Y(y,y'),\pi\}$.

For a vertex $v$ in a cubed complex $X$,
we denote the $0$-cone $C_0(\Lk(v,X))$ 
by $T_v(X)$ 
and call it the {\it tangent cone} at $v$~\cite[Definition II.3.18]{BH}.
}
\end{definition}

We have the following fundamental relation between 
CAT(0) spaces and CAT(1) spaces,
where the second statement (Gromov's link condition)
is proved by using the first statement (Berestovskii's theorem).

\begin{proposition}
\label{thm:CAT01}
{\rm (1) (Berestovskii~\cite[Theorem II.3.14]{BH})} 
Let $Y=(Y,d_Y)$ be a metric space.
Then the $0$-cone $C_0(Y)$ over $Y$ is a CAT(0) space 
if and only if $Y$ is a CAT(1) space.

{\rm (2) (Gromov's link condition)~\cite[Theorem II.5.20]{BH}}
A finite dimensional cubed complex $X$ is non-positively curved
if and only if, 
for every vertex $v\in X$, the geometric link $\Lk(v,X)$ is a CAT(1) space.
\end{proposition}

The following lemma is a simple consequence of the above results.

\begin{lemma}
\label{lem:W-CAT(0)}
Let $X$ be a finite dimensional
CAT(0) cubed complex and $W$ a connected subcomplex of $X$. 
Then the following hold.

{\rm (1)} For a vertex $v$ of $W$, if 
$\Lk(v,W)$ is a full subcomplex of $\Lk(v,X)$
then the tangent cone $T_v(W)$ is a CAT(0) space.

{\rm (2)}
If $\Lk(v,W)$ is a full subcomplex of $\Lk(v,X)$
for every vertex $v$ of $W$,
then the cubed complex $W$ is non-positively curved.
\end{lemma}

\begin{proof}
(1) Since $X$ is a CAT(0) cubed complex,
$\Lk(v,X)$ is a flag complex by Proposition~\ref{thm:CAT01}(2).
If $\Lk(v,W)$ is a full subcomplex of $\Lk(v,X)$,
then $\Lk(v,W)$ is also a flag complex.
So, the all-right spherical complex $\Lk(v,W)$ is CAT(1)
by Proposition~\ref{thm:CAT1}(2).
Hence, $T_v(W)$ is a CAT(0) space by Proposition~\ref{thm:CAT01}(1).

(2) is proved by a similar argument
by using Proposition~\ref{thm:CAT01}(2) instead of Proposition~\ref{thm:CAT01}(1) in the last step.
\end{proof}

Next, we prove the following key lemma for 
the proof of Theorem~\ref{prop:local-convex}.

\begin{lemma}
\label{lem:pi-convex}
Let $L=(L,d)$ be a finite dimensional all-right spherical complex 
that is a flag complex,
and let $K$ be a subcomplex of $L$.
Then the following conditions are equivalent.
\begin{enumerate}[\rm (1)]
\item
$K$ is $\pi$-convex in $L$,
namely, for any points $a$ and $b$ of $K$ with $d(a,b)<\pi$,
the unique geodesic segment $[a,b]$ in $L$ is contained in $K$.
\item
$K$ is a full subcomplex of $L$.
\end{enumerate}
\end{lemma}

\begin{proof}
We first prove that (1) implies (2).
Suppose that 
$K$ is not full in $L$.
Then there is a simplex $\sigma$ of $L\backslash K$
such that $\partial \sigma$ is contained in $K$.
Pick a vertex $v$ of $\sigma$, and let $\tau$ be the codimension $1$ face of $\sigma$
that does not contain the vertex $v$.
Pick a point $y$ in the interior of $\tau$.
Then $d(v,y)=\pi/2$ and the interior of the geodesic segment $[v,y]$
is contained in the interior of $\sigma$.
Thus $[v,y]$ is not contained in $K$ 
though both $v$ and $y$ are contained in $K$.
Hence $K$ is not $\pi$-convex.

We next prove that (2) implies (1).
Suppose to the contrary that $K$ is not $\pi$-convex
though $K$ is a full subcomplex of $L$. 
Then there is a geodesic segment $[a,b]$ in $L$ of length $<\pi$
such that $a,b\in K$ but $[a,b]\not\subset K$.
If necessary, by replacing $[a,b]$ with a sub geodesic segment, 
we may assume $K\cap [a,b]=\{a,b\}$.
Let $\sigma$ be the simplex of $L$ 
whose interior intersects the germ of $[a,b]$ at $a$.
Then $\sigma$ is not a simplex of $K$.
Since $K$ is a full subcomplex of $L$ by the assumption,
there is a vertex $v$ of $\sigma$ that is not contained in $K$.
Let $\cst(v,L)$ (resp.\ $\ost(v,L)$) be the {\it closed star} (resp.\ {\it open star})
of $v$ in $L$,
i.e., the union of the simplices (resp.\ the interior of the simplices) of $L$
that contain $v$.
Note that $\cst(v,L)=\ost(v,L)\sqcup \clk(v,L)$,
where $\clk(v,L)$ is the {\it simplicial link} of $v$ in $L$,
i.e., the union of the simplices $\tau$ of $L$ 
such that $v\notin \tau$ and $
\{v\}\cup \tau$ is contained in a simplex of $L$.
Then $\ost(v,L) \cap K=\emptyset$ and therefore there is a point $b'\in (a,b]$
such that $b'\in \clk(v,L)$ and $(a,b')\subset \ost(v,L)$.

Case 1. $v\in (a,b')$.
Then $d(v,a)=d(v,b')=\pi/2$ and hence
$d(a,b)\ge d(a,b')=d(a,v)+d(v,b')=\pi$,
a contradiction.

Case 2. $v\notin (a,b')$.
We consider the \lq\lq development'' of $[a,b']\subset \cst(v,L)$ 
in the northern hemisphere $S^2_+$,
the closed ball of radius $\pi/2$ centered at the north pole $N=(0,0,1)$
in $S^2$,
that is defined as follows (cf.~\cite[Definition I.7.17]{BH}).
Let $a=y_0, y_1, \cdots, y_n=b'$ be points lying in $[a,b']$
in this order,
such that $(y_{i-1}, y_i)$ is contained in 
the interior of a simplex $\sigma(i)$ of $L$ for each $i$ ($1\le i\le n$).
Note that $\sigma(i)$ contains $v$ as a vertex.
Let $\bar{y}_0=(1,0,0), \bar{y}_1, \cdots, \bar{y}_n$ be the points in $S^2_+$ satisfying the
following conditions.
\begin{enumerate}
\item
$d_{S^2}(N,\bar{y}_i)=d_{\sigma(i)}(v,y_i)=d(v,y_i)$
and $d_{S^2}(\bar{y}_{i-1},\bar{y}_i)=d_{\sigma(i)}(y_{i-1},y_i)=d(y_{i-1},y_i)$
for each $i$.
\item
If $N, \bar{y}_{i-1},\bar{y}_{i}$ are not aligned,
the initial vectors of the geodesic segments $[N,\bar{y}_{i-1}]$ and $[N,\bar{y}_{i}]$
in $S^2_+$ occur in the order
of a fixed orientation of $S^2$.
\end{enumerate}
We call the union $\gamma:=\cup_{i=1}^n[\bar{y}_{i-1},\bar{y}_i]\subset S^2_+$ 
the {\it development} of $[a,b']\subset \cst(v,L)$ in $S^2_+$.
It should be noted that $n \ge 2$ and 
$\bar{y}_1, \cdots, \bar{y}_{n-1}$ are contained in $\interior S^2_+$.

\begin{claim}
$\gamma$ is a local geodesic in $S^2$.
\end{claim}

\begin{proof}
Though this is used without proof in
~\cite[the 4th paragraph in the proof of Theorem II.5.18]{BH}, 
we give a proof for completeness. 
If $\gamma$ is not a local geodesic, then 
$\length([\bar{y}_{i-1},\bar{y}_i]\cup [\bar{y}_i,\bar{y}_{i+1}])
>\length([\bar{y}_{i-1},\bar{y}_{i+1}])$ for some $i$.
Let $\bar{y}_i'$ be the intersection of 
the geodesic segment $[\bar{y}_{i-1},\bar{y}_{i+1}]$
and the maximal geodesic segment in $S^2_+$ emanating from $N$ and passing through $\bar{y}_i$.
Let $y_i'$ be the point in the maximal geodesic segment in $\sigma(i)\cap\sigma(i+1) \subset L$ emanating from $v$ and 
passing through $y_i$, such that $d(v,y_i')=d_{S^2}(N,\bar{y}_i')$.
Then we have the following isometries 
among spherical triangles.
\[
\Delta(v,y_{i-1},y_{i}')\cong \Delta(N,\bar{y}_{i-1},\bar{y}_i'),
\quad
\Delta(v,y_{i}',y_{i+1})\cong \Delta(N,\bar{y}_i',\bar{y}_{i+1})
\]
Hence the following hold.
\begin{align*}
\length([y_{i-1},y_{i}']\cup[y_{i}',y_{i+1}])
& =  \length([\bar{y}_{i-1},\bar{y}_i']\cup [\bar{y}_i',\bar{y}_{i+1}])\\
& =  \length([\bar{y}_{i-1},\bar{y}_{i+1}])\\
& <  
\length([\bar{y}_{i-1},\bar{y}_i]\cup [\bar{y}_i,\bar{y}_{i+1}])\\
& = 
\length([y_{i-1},y_{i}]\cup[y_{i},y_{i+1}])
= 
\length([y_{i-1},y_{i+1}])
\end{align*}
This contradicts the fact that 
$[y_{i-1},y_{i+1}] (\subset [a,b']\subset [a,b])$ is a geodesic.
\end{proof}

Since $\gamma$ is a local geodesic with length $\length (\gamma)<\pi$,
it is a geodesic in $S^2_+$ by~\cite[Proposition II.1.4(2)]{BH}.
Since $y_n=b'\in \clk(v,L)$, we see
$d(v,y_n)=\pi/2$ and so  $\bar{y}_n\in \partial S^2_+$.
Thus the endpoints $\bar{y}_0$ and $\bar{y}_n$
of the geodesic $\gamma\subset S^2_+$ are contained in $\partial S^2_+$.
Since $\length(\gamma) < \pi$, this implies $\gamma\subset \partial S^2_+$.
This contradicts the fact that 
$\bar{y}_1, \cdots, \bar{y}_{n-1}$ are contained in $\interior S^2_+$.
This completes the proof of Lemma~\ref{lem:pi-convex}.
\end{proof}

In addition to Lemma~\ref{lem:pi-convex},
we need Lemma~\ref{lem:relative-versions} below
which gives relative versions of two results included in~\cite{BH}
concerning the local shape of polyhedral complexes.

\begin{notation}
{\rm
For a vertex $v$ of a subcomplex $W$ of a cubed complex $X$,
the symbol $j:T_v(W)\to T_v(X)$ denotes the natural injective map
from the tangent cone $T_v(W)$ of the cubed complex $W$ 
into the tangent cone $T_v(X)$ of the cubed complex $X$. 
Note that $j$ is not necessarily an isometric embedding.
}
\end{notation}

\begin{lemma}
\label{lem:relative-versions}
Let $X$ be a finite dimensional cubed complex and $W$ a subcomplex of $X$.
Then the following hold. 

{\rm (1) (Relative version of~\cite[Theorem I.7.39]{BH})} 
Let $v$ be a vertex of $W$.
Then there is a natural isometry $\varphi$ from the open ball 
$B_X(v,1/2)$ in $X$
onto the open ball $B_{T_v(X)}(0,1/2)$ in the tangent cone $T_v(X)$ 
that carries $W\cap B_X(v,1/2)$ onto 
$j(T_v(W))\cap B_{T_v(X)}(0,1/2)$.

{\rm (2) (Relative version of~\cite[Lemma I.7.56]{BH})} 
Let $x$ and $y$ be points of $W$ contained in the same open cell of $W$.
Then, for sufficiently small $\epsilon>0$,
there exists a natural isometry between the open balls 
$B_X(x,\epsilon)$ and $B_X(y,\epsilon)$
in $X$
that carries $W\cap B_X(x,\epsilon)$ onto $W\cap B_X(y,\epsilon)$.
\end{lemma}

\begin{proof}
(1) By~\cite[Theorem I.7.39]{BH},
there is
a natural isometry from $B_X(v,1/2)$ onto $B_{T_v(X)}(0,1/2)$.
(The radius $1/2$ is the half of the length $1$ of the unit interval $I$,
and it corresponds to $\varepsilon(x)/2$ in~\cite[Theorem I.7.39]{BH}.)
The isometry is defined as follows (see~\cite[the first paragraph 
in the proof of Theorem I.7.16 in p.104]{BH}). 
If $x \in B_X(v,1/2)$ then there is a direction $u\in \Lk(v,X)$
such that $x$ is a distance $t<1/2$ along the geodesic issuing from $v$ in the direction $u$. 
(Here $u$ is uniquely determined by $x$ except when $x=v$,
i.e., $t=0$.)
Then $x \in B_X(v,1/2)$  is mapped to the point $tu\in B_{T_v(X)}(0,1/2)$.
By this definition of the isometry, 
we see that it carries $W\cap B_X(v,1/2)$ onto 
$j(T_v(W))\cap B_{T_v(X)}(0,1/2)$.

(2) By~\cite[Lemma I.7.56]{BH}, there is a natural isometry 
from $B_X(x,\epsilon)$ onto $B_X(y,\epsilon)$
that restricts to an isometry from 
$C\cap B_X(x,\epsilon)$ onto $C\cap B_X(y,\epsilon)$
for every closed cell $C$ of $X$ containing $x$ and $y$.
Obviously the isometry carries 
$W\cap B_X(x,\epsilon)$ onto $W\cap B_X(y,\epsilon)$.
\end{proof}

We now give a proof of the main Theorem~\ref{prop:local-convex}.

\begin{proof}[Proof of Theorem~\ref{prop:local-convex}]
Let $X$ be a finite dimensional
CAT(0) cubed complex and $W$ a subcomplex of $X$. 
Assume that $W$ is
combinatorially locally convex in $X$, i.e.,
$\Lk(v,W)$ is a full subcomplex of $\Lk(v,X)$
for every vertex $v$ of $W$.
Then we have the following claim.

\begin{claim}
\label{claim:convex-cone}
For any vertex $v$ of $W$,
the map $j:T_v(W)\to T_v(X)$ is an isometric embedding, and 
$j(T_v(W))$ is convex in $T_v(X)$.
\end{claim}

\begin{proof}
Let $v$ be a vertex of $W$.
Then, by the assumption and
Lemma~\ref{lem:pi-convex},
$\Lk(v,W)$ is $\pi$-convex in $\Lk(v,X)$.
This implies that the distance function $d_{\Lk(v,W)}^{\pi}$
on $\Lk(v,W)$
is equal to the restriction of 
the distance function
$d_{\Lk(v,X)}^{\pi}$ on $\Lk(v,X)$
to the subspace $\Lk(v,W)$.
Hence
$j:T_v(W)\to T_v(X)$
is an isometric embedding.
On the other hand, 
$T_v(W)$ is a CAT(0) space by Lemma~\ref{lem:W-CAT(0)}(1).
Hence, any two points of $T_v(W)$ are 
joined by a unique geodesic
in the metric space $T_v(W)$.
Its image in $T_v(X)$ is also a geodesic in the metric space $T_v(X)$,
because $j:T_v(W)\to T_v(X)$
is an isometric embedding.
Hence $j(T_v(W))$ is convex in $T_v(X)$ as desired.
\end{proof}

Now let $x$ be an arbitrary point in $W$.
Pick a vertex $v$ of the open cell of $W$ that contains $x$.
Then, by Lemma~\ref{lem:relative-versions}(2), we can find
a small real $\epsilon>0$ and  $x'\in B_X(v,1/2)$ with 
$B_X(x',\epsilon)\subset B_X(v,1/2)$, such that
$(B_X(x,\epsilon),W\cap B_X(x,\epsilon))$ is isometric to
$(B_X(x',\epsilon),W\cap B_X(x',\epsilon))$.
Recall the following isometry given by 
Lemma~\ref{lem:relative-versions}(1).
\[
\varphi:(B_X(v,1/2), W\cap B_X(v,1/2))\to
(B_{T_v(X)}(0,1/2), j(T_v(W))\cap B_{T_v(X)}(0,1/2))
\]
Since $B_X(x',\epsilon)\subset B_X(v,1/2)$,
we have the following identities.
\begin{align*}
\varphi(B_X(x',\epsilon))&=B_{T_v(X)}(\varphi(x'),\epsilon),\\
\varphi(W\cap B_X(x',\epsilon))&=j(T_v(W))\cap B_{T_v(X)}(\varphi(x'),\epsilon).
\end{align*}
Since $j(T_v(W))$ is convex in $T_v(X)$ 
by Claim~\ref{claim:convex-cone}
and since $B_{T_v(X)}(\varphi(x'),\epsilon)$ is 
convex in the CAT(0) space $T_v(X)$ by~\cite[Proposition II.1.4(3)]{BH}, 
these identities imply that
$\varphi(W\cap B_X(x',\epsilon))$ is convex in the convex subset 
$B_{T_v(X)}(\varphi(x'),\epsilon)$ of $T_v(X)$.
Since we have the isometries
\begin{align*}
(B_X(x,\epsilon),W\cap B_X(x,\epsilon))
&\cong
(B_X(x',\epsilon),W\cap B_X(x',\epsilon))\\
&\cong
(\varphi(B_X(x',\epsilon)),\varphi(W\cap B_X(x',\epsilon))),
\end{align*}
this in turn implies that 
$W\cap B_X(x,\epsilon)$ is convex in the convex subset $B_X(x,\epsilon)$
of $X$.
Hence $W\cap B_X(x,\epsilon)$ is convex in $X$,
completing the proof of the if part of Theorem~\ref{prop:local-convex}.

Though the only if part of Theorem~\ref{prop:local-convex}
may be trivial, we include a proof for completeness.
Suppose that $\Lk(v,W)$ is not a full subcomplex of $\Lk(v,X)$.
Then $\Lk(v,W)$ is not $\pi$-convex by Lemma~\ref{lem:pi-convex}, and so
there is a geodesic segment $[a,b]$ in $\Lk(v,X)$
such that $[a,b]\cap\Lk(v,W)=\{a,b\}$. 
Pick a small $t>0$ so that 
the geodesic $[ta,tb]$ in $T_v(X)$ is contained in the 
open ball $B_{T_v(X)}(0,1/2)$.
(In fact, any positive $t<1/2$ works.)
Since the geodesic $[ta,tb]$ intersects $j(T_v(W))$
only at the endpoints,
the inverse image of $[ta,tb]$ by the isometry
$\varphi$ in  Lemma~\ref{lem:relative-versions}(1)
is
a geodesic in $B_X(v,1/2)$ 
that intersects $W$ only at the endpoints.
Hence $W$ is not locally convex.
\end{proof} 


\section{Late additions}
\label{sec:late-addition}

Immediately after submission of the first version of this note
to the arXiv,
we were informed by Ian Leary that
his paper \cite{Leary} is relevant to the note.
In fact, Appendix B of the paper includes 
a very simple proof of the main theorem, Theorem~\ref{thm:convex}.
Moreover, he extends the theorem to the infinite dimensional case.
His aim in that appendix,
suggested by Michah Sageev, was to establish Gromov's 
flag criterion~\cite[Theorem II.5.20]{BH} 
for infinite dimensional cubical complexes,
and he needed the above result in his argument.
He also informed us that
Yael Algom-Kfir gave a proof of 
Gromov's flag criterion for infinite dimensional cubical complexes 
in her master thesis 
under the supervision of Sageev.

Shortly after correspondence with Leary,
we learned from Takuya Katayama that
the main technical result, Theorem~\ref{prop:local-convex},
which plays a key role in the proof of Theorem~\ref{thm:convex},
immediately follows from 
the result \cite[Theorem 1(2)]{Crisp-Wiest} due to Crisp and Wiest.
He also informed us of the paper~\cite{Farley} by Farley,
which includes a proof of the fact that a hyperplane in a CAT(0) cubed complex $X$
is convex and that it divides $X$ into two convex subspaces.
(To be precise, it is an immediate consequence
of \cite[Section 4]{Farley}, and it is also proved by a slight 
modification of \cite[Proof of Theorem 4.4]{Farley}.
The fact plays a key role in our paper \cite{SS}.
The desire to give its proof  
was the real motivation of this note,
because we did not know a reference for that fact
though we knew its combinatorial version 
established by Sageev \cite[the first line in p.612]{Sageev}).

We thank Ian Leary and Takuya Katayama for these invaluable informations.
Though we do not intend to submit this note to a journal,
we upload this revised version to the arXiv for our record.
In this added section of the revised version, we give 
(i) a simple proof of the main result 
following the arguments by 
Leary~\cite[Proof of Theorems B.7 and B.9]{Leary}, 
(ii) a proof of an assertion in the proof in  
Crisp-Wiest~\cite[Theorem 1(2)]{Crisp-Wiest}
whose proof is omitted, 
and 
(iii) 
a proof of the convexities of
hyperplanes and the half-spaces bounded by hyperplanes in CAT(0) cubed complexes by using 
Theorem~\ref{thm:convex},
and its much simpler proof by using 
Farley's work \cite{Farley}.

\subsection{A simple proof of Theorem~\ref{thm:convex}
due to Leary}

As we note in the above, it turned out that
a very simple proof of Theorem~\ref{thm:convex}
(and so that of Theorem~\ref{prop:local-convex})
had already been given by Leary~\cite[Theorem~B.9]{Leary}.
The key idea is to consider the double of $X$ along $W$.
We include the proof following his arguments in
\cite[Proof of Theorems B.7 and B.9]{Leary}.

\begin{proof}[Proof of 
Theorem~\ref{thm:convex} following Leary~\cite{Leary}]
Since the only if part is obvious, we prove the if part.
Let $X *_W X$ be the double of $X$ along $W$, 
that is, the quotient space obtained from two copies of $X$
by identifying the two copies of $W$.
Each of the two inclusions $X\to X *_W X$ is a map
that does not increase distance
(since a rectifiable path in $X$
is mapped to a rectifiable path in $X *_W X$
of the same length).
The composite map 
\[
X \to X *_W X \to X *_X X=X
\]
is the identity, and hence each inclusion map
$X \to X *_W X$ is an isometric embedding.
Moreover, there is an isometric involution of $X *_W X$
swapping the copies of $X$
whose fixed point set is $W$.

It is easily seen that
the assumption that $W$ is combinatorially locally convex 
implies that
the link of each vertex of $X *_W X$
is a flag complex.
By Gromov's flag criterion~\cite[Theorem II.5.20]{BH},
it follows that $X *_W X$ is non-positively curved.
Since $X *_W X$ is simply-connected by van Kampen's theorem,
this implies that $X *_W X$ is a CAT(0) space by
the Cartan-Hadamard theorem~\cite[Theorem II.4.1(2)]{BH}.

Given any two points $x, y\in W$,
consider the unique geodesic $[x,y]$ in the CAT(0) space $X$.
Then its image in $X *_W X$
by each of the isometric embeddings is 
also a geodesic in  $X *_W X$.
If $[x,y]$ did not lie entirely inside $W$,
its image by the isometric involution
gives another geodesic in 
the CAT(0) space $X *_W X$ joining $x$ to $y$,
a contradiction.
Hence $[x,y]$ does lie entirely inside $W$, and so
$W$ is convex in $X$.
\end{proof}

Leary actually establishes an infinite dimensional versions of 
Theorem~\ref{thm:convex} and 
Gromov's flag criterion.
See his very careful and reader-friendly 
treatment~\cite[Appendices B and C]{Leary} for the details.

\subsection{A supplementary to Crisp-Wiest~\cite[Theorem 1(2)]{Crisp-Wiest}}
\label{subsec:2}

Theorem~\ref{prop:local-convex} is a direct consequence of 
the following theorem due to 
Crisp and Wiest~\cite{Crisp-Wiest}.

\begin{theorem}
\label{thm:CW}
{\rm \cite[Theorem 1(2)]{Crisp-Wiest}}
Let $X$ and $Y$ be finite dimensional cubed complexes and 
$f:X\to Y$ a cubical map. 
Suppose that $Y$ is locally CAT(0). Then the map $f$ is locally an isometric embedding if and only if, for every vertex 
$v\in X$, the simplicial map $f_v:\Lk(v, X) \to \Lk(f(v), Y)$ 
induced by $f$ is injective with image a full subcomplex of 
$\Lk(f(v),Y)$.
\end{theorem}
Here a {\it cubical map} is a continuous map
induced by a combinatorial map which takes 
the interior of each cube onto that of a cube
of the same dimension locally isometrically.
See~\cite[p.443]{Crisp-Wiest} for a precise definition,
and see~\cite[Definition 2.9]{Farley} 
for a simpler definition when
$X$ and $Y$ are cubical complexes.
The proposition implicitly assumes that $\Lk(v, X)$
is a simplicial complex, and we do assume this in the remainder.
If $f:X\to Y$ is equal to the inclusion map $i:W\to X$
under the setting of Theorem~\ref{prop:local-convex},
then the assumption is certainly satisfied.

Theorem~\ref{thm:CW} 
is proved in~\cite{Crisp-Wiest} 
as a consequence of the following three assertions.
\begin{enumerate}
\item
$f:X\to Y$ is locally an isometric embedding
if and only if, for every point $x\in X$, 
the map $f_x:S_x(X)\to S_{f(x)}(Y)$
between the space of directions induced by $f$
is {\it $\pi$-distance preserving}, i.e.,
$d_{S_{f(x)}(Y)}(f(u),f(u'))=\pi$ whenever $d_{S_x(X)}(u,u')=\pi$.
\item
The latter condition in (1) holds if and only if,
for every vertex $v\in X$,
the map $f_v:S_v(X)\to S_{f(v)}(Y)$ 
is {\it $\pi$-distance preserving}.
\item
The latter condition in (2) holds if and only if,
for every vertex $v\in X$,
the simplicial map $f_v:\Lk(v, X) \to \Lk(f(v), Y)$ 
induced by $f$ is injective with image a full subcomplex of 
$\Lk(f(v),Y)$.
\end{enumerate}
The assertion (1) directly follows from
Charney~\cite[Lemma 1.4]{Charney}.
The assertion (3) corresponds to
the key Lemma~\ref{lem:pi-convex} in this note.
As noted in \cite[the first paragraph in p.444]{Crisp-Wiest},
the proof of (3) (and our proof of Lemma~\ref{lem:pi-convex}) 
are based on the fact that
any locally geodesic segment in the closed half hemi-sphere $S^2_+$  
which has endpoints in $\partial S^2_+$ and intersect $\interior S^2_+$
has length at least $\pi$,
which in turn is a key element of Gromov's proof of 
his flag condition for non-positively curved cubed complexes.

For the assertion (2), Crisp-Wiest~\cite{Crisp-Wiest} only writes
that it is an easy consequence of~\cite[Lemmas 1.4 and 1.5]{Charney}.
Katayama taught us a detailed proof of (2),
which is essentially a refinement of Lemma~\ref{lem:relative-versions}.
We present another proof below. 

Let $x$ be a point of the cubed complex $X$ in
Theorem~\ref{thm:CW},
and let $C$ be the cell of $X$
whose relative interior contains $x$. 
Let $\Lk(C,X)$ be the subspace of $S_x(X)$,
the space of directions of $X$ at $x$,
consisting of those directions that are normal to $C$.
(If $C$ is a vertex, $\Lk(C,X)$ is defined to be $\Lk(v,X)$.) 
By the (implicit) assumption that $\Lk(v,X)$ is a simplicial complex
for every vertex $v$, we see that 
$\Lk(C,X)$ is naturally regarded as an
(all-right spherical) simplicial complex,
and the following hold.
\begin{enumerate}
\item[(a)]
If $k:=\dim C>0$, then
$S_x(X)$ is isometric to the spherical join $S^{k-1}*\Lk(C,X)$,
where $S^{k-1}$ is the unit
$(k-1)$-sphere (see~\cite[Lemma 2.5]{Charney-Davis}).
\item[(b)]
For a vertex $v$ of $C$, we have the following isomorphism
between simplicial complexes.
\[
\Lk(C,X)\cong\Lk(\Lk(v,C),\Lk(v,X))
\]
Here $\Lk(\Lk(v,C),\Lk(v,X))$ is the (simplicial) link
of the simplex $\Lk(v,C)$ of the simplicial complex $\Lk(v,X)$.  
\end{enumerate}

The assertion (2) is proved as follows.
By (a) and \cite[Lemma 1.5]{Charney},
$f_x:S_x(X)\to S_{f(x)}(Y)$
is $\pi$-distance preserving
if and only if 
$f_x:\Lk(C,X)\to \Lk(f(C),Y)$ is $\pi$-distance preserving,
where the second $f_x$ is the restriction of the first one
to the subspace $\Lk(C,X)\subset S_x(X)$.
By \cite[Lemma 1.6]{Charney}, the latter condition holds
if and only if $f_x:\Lk(C,X)\to \Lk(f(C),Y)$ is injective
and its image $f_x(\Lk(C,X))$ is a full subcomplex of $\Lk(f(C),Y)$.
By (b), the last condition holds for every $x\in X$,
if and only if, for every vertex $v$ and for every cell $C$ of $X$
containing $v$,
(i) the simplicial map $f_v:\Lk(v, X) \to \Lk(f(v), Y)$ 
is injective, 
(ii) $f_v(\Lk(v, X))$ is full in $\Lk(f(v), Y)$,
and (iii) $f_v(\Lk(C,X))$ is full in $\Lk(f(C), Y)$.
We can easily see that (iii) is 
a consequence of (i) and (ii).
Hence, we obtain the assertion (2).

\subsection{
Convexities of hyperplanes and their complementary halfspaces
in CAT(0) cubical complexes}

We give a proof of the following theorem
by using Theorem~\ref{thm:convex}.

\begin{theorem}
\label{thm:hyperplane}
Let $X$ be a finite dimensional CAT(0) cubical complex
and $\Sigma$ a hyperplane in $X$.
Then $\Sigma$ is convex in $X$.
Moreover, $\Sigma$ divides $X$ into two closed convex subspaces.
\end{theorem}

Recall that a {\it cubical complex} is a cubed complex such that
each cell is isometric to a cube $I^{n_{\lambda}}$
and that the link of every vertex is a simplicial complex~
\cite[Example I.7.40(3)]{BH}.
By~\cite[Theorem C.4]{Leary}, every CAT(0) cubed complex is cubical.
A {\it hyperplane} in a CAT(0) cubical complex $X$ 
is a subspace of $X$ which is obtained as the union of 
a family of midcubes of cells (cubes) in $X$,
that satisfies a certain maximality condition.
(Here a {\it midcube} of a cube $I^n$
is the subspace of the form $I^{n_1}\times\{1/2\}\times I^{n_2}$
with $n_1+n_2=n-1$.)
For a precise definition of a hyperplane,
see~\cite[Definition 4.5]{Farley}
(cf.~\cite[Section 2.4]{Sageev},
\cite[Definition 2.2]{Haglund-Wise}).
The following fundamental theorem is proved by
Sageev~\cite[Theorem 4.10]{Sageev}.

\begin{theorem}
\label{thm:Sageev}
Suppose $X$ is a CAT(0) cubical complex and $\Sigma$ is a hyperplane in $X$.
Then $\Sigma$ does not self-intersect,
namely, for each cube of $X$, its intersection with $\Sigma$ 
is either empty or a single midcube.
Moreover, $X\setminus\Sigma$ has exactly two components.
\end{theorem}  

\begin{proof}[Proof of Theorem~\ref{thm:hyperplane}]
Let $X'$ be the first cubical subdivision of $X$.
Then the hyperplane $\Sigma$ is regarded as a subcomplex of $X'$.
Let $v$ be a vertex of $\Sigma\subset X'$.
Then, by using the first assertion of Theorem~\ref{thm:Sageev},
we see that 
$\Lk(v,X')$ is the spherical join (or the spherical suspension)
$S^0*\Lk(v,\Sigma)$. 
Hence $\Lk(v,\Sigma)$ is a full subcomplex of $\Lk(v,X')$.
By Theorem~\ref{thm:convex}, this implies that
$\Sigma$ is convex in $X'$, completing the proof of the first assertion.

To prove the second assertion, 
recall the second assertion of Theorem~\ref{thm:Sageev}, and 
let $X_1$ and $X_2$ be the closures of the components of 
$X'\setminus \Sigma$.
Regard $X_1$ and $X_2$ as subcomplexes of $X'$,
Then we can easily check that the link of each vertex $v$
of each of $X_1$ and $X_2$ is a full subcomplex of $\Lk(v,X')$.
Hence, $X_1$ and $X_2$ are convex in $X'$ by Theorem~\ref{thm:convex},
completing the proof of the second assertion.
\end{proof}

Finally, we recall Farley's results in \cite{Farley}
that immediately imply Theorem~\ref{thm:hyperplane}.
For every 
finite dimensional CAT(0) cubical complex $X$,
Farley constructs a cubical complex $\BB(X)$
and a map $\pi_{\BB}:\BB(X)\to X$ that satisfy the following conditions
for every component $B$ of $\BB(X)$.
\begin{enumerate}
\item
The map $\pi_{\BB}$ embeds $B$ isometrically into $X$
(\cite[Theorem 4.1]{Farley}). 
\item
There is a hight function $h:B\to [0,1]$,
such that, for each $t\in [0,1]$,
$B_t:=h^{-1}(t)$ is a closed convex set of $\BB(X)$.
The space $\pi_{\BB}(B_t)$ is a closed convex subset of $X$
(\cite[lemma 4.3(1)]{Farley}).
Moreover, $\pi_{\BB}(B_{1/2})$ is a hyperplane
(\cite[Definition 4.5]{Farley}).
\item
$\pi_0\times h:B \to B_0\times [0,1]$ is an isometry,
where $\pi_0$ is the projection onto the closed convex subspace $B_0$
(\cite[lemma 4.3(2)]{Farley}).
\item
Each subspace $\pi_{\BB}(B_t)$ ($0<t<1$)
separates $X$ into two open convex complementary half-spaces
(\cite[Theorem 4.3]{Farley}).
\end{enumerate}
Theorem~\ref{thm:hyperplane} immediately follows from
these results.
In fact, the convexity of a hyperplane
is included in (2),
and the convexity of closed half-spaces bounded by a hyperplane 
follows from the fact that
every such closed half-space is the intersection of 
a family of convex open half-spaces in (4). 
It can be also proved directly by a slight modification
of the final step of \cite[Proof of Theorem 4.4]{Farley}.

We note that the assertion (1) 
is obtained by a nice application of
Crisp-Wiest~\cite[Theorem 1(2)]{Crisp-Wiest},
which we explained in Subsection~\ref{subsec:2}.
Moreover, he observes that 
the result of Crisp-Wiest 
and so Theorem~\ref{thm:hyperplane} hold 
under the weaker assumption that 
the cubical complex $X$ is {\it locally finite-dimensional},
i.e., the link of each vertex is a finite-dimensional 
simplicial complex~\cite[Definition 2.1]{Farley}.  
According to \cite[Theorem A.6]{Leary}, 
this is a necessary and sufficient condition
for a cubical complex to be complete.


\end{document}